\documentclass[12pt]{amsart}
\usepackage{amsmath,amsthm,amsfonts,amssymb}
\usepackage{color}
\usepackage{amsfonts,amsthm,amsmath,amssymb,verbatim}
\usepackage{graphicx}
\usepackage[active]{srcltx}
\usepackage[all,cmtip]{xy}
\usepackage{tikz}
\usepackage{xypic}
\usepackage{hyperref}
\hypersetup{
    colorlinks=true,
   linkcolor=blue,
   filecolor=magenta,
    urlcolor=cyan,
}

\newtheorem{thm}{Theorem}[section]
\newtheorem{prop}[thm]{Proposition}
\newtheorem{lemma}[thm]{Lemma}

\theoremstyle{remark}

\newtheorem{example}[thm]{Example}

\newtheorem{defin}{Definition}



\def\C{\mathbb{C}}

\def\Q{\mathbb{Q}}
\def\Z{\mathbb{Z}}
\def\P{\mathbb{P}}

\def\F{\mathbb{F}}
\def\A{\mathbb{A}}

\def\G{\Gamma}
\def\eps{\varepsilon}

\title{Quadratic residue patterns, algebraic curves and a K3 surface}

\author[V. Kiritchenko]{Valentina Kiritchenko$^{1}$}
\address{$^1$HSE University,
Moscow, Russia}
\address{$^2$Higher School of Modern Mathematics MIPT, Russia}
\address{$^3$Aix Marseille Universit\'{e}, CNRS,  I2M UMR 7373, Marseille, France}
\address{$^4$University of California, Berkeley, USA}
\author[M. Tsfasman]{Michael Tsfasman$^{2}$}
\author[S. Vl\u{a}du\c{t}]{Serge Vl\u{a}du\c{t}$^{3}$}
\author[I. Zakharevich]{Ilya Zakharevich$^{4}$}
\thanks {The research of the second named author is supported by the MSHE project No. FSMG-2024-0048}

\email{vkiritch@hse.ru}
\email{mtsfasman@yandex.ru}
\email{sergevladuts@ya.ru}
\email{ilya@math.berkeley.edu}


\begin{document}

\dedicatory{\flushright To the memory of Lydia Goncharova}

\maketitle

\begin{abstract}Quadratic residue patterns modulo a prime are studied since 19th century. In the first part we extend existing results on the number of consecutive $\ell$-tuples of quadratic residues, studying corresponding algebraic curves and their  Jacobians, which happen to be products of  Jacobians of hyperelliptic curves. In the second part we state the last unpublished result of Lydia Goncharova on squares such that their differences are also squares, reformulate it in terms of algebraic geometry of a K3 surface, and prove it. The core of this theorem is an unexpected relation between the number of points on the K3 surface and that on a CM elliptic curve.
\end{abstract}

\section{Introduction}                
Patterns formed by quadratic residues and non-residues modulo a prime have been studied since the end of the XIX century \cite{A,St,J} and continue to attract attention of contemporary mathematicians \cite{C,MT}.
Even the most elementary questions about distributions of quadratic (non)residues lead to difficult problems and deep results of number theory.
For instance, given a positive integer $\ell$, does there exist a prime number $p$ and an $\ell$-tuple of consecutive integers between $1$ and $p-1$ such that the $\ell$-tuple consists only of quadratic non-residues (or only of residues) modulo $p$? How many such $\ell$-tuples are there for given $\ell$ and $p$ ? 
And how large is the least quadratic non-residue modulo $p$? Another possibility is, e.g., to ask about  the number of quadratic residues such that their  differences are also quadratic residues. 

In the first part of this note we study the following question. For a given $\ell$ and given $p$ can we calculate the number of $\ell$-tuples of consecutive quadratic residues? Is there an explicit formula for that? If not, can we describe its asymptotic behaviour for growing $p$? Classical results in this direction were proved using sums of Legendre symbols. Our approach is that of algebraic geometry. It permits us to reprove known results easily, to go further in $\ell$, and to obtain finer asymptotic results. Our luck is that corresponding complete intersections of quadrics are very specific, they have a basis consisting of quadrics of rank 3, and even stronger, the determinant variety of the corresponding net of quadrics is a union of hyperplanes. We derive that the  Jacobians of corresponding curves up to an isogeny are always products of  Jacobians of hyperelliptic curves (Theorem \ref{jac}), and in some first cases just of elliptic ones.

In the second part we formulate a new result (Theorem \ref{t.main}) on patterns of quadratic (non)residues obtained by our late friend Lydia Goncharova.
We recovered her result partly from her talk with the first author in December 2019, and partly from her older notes and e-mail correspondence.
We are deeply grateful to David Kazhdan who helped us to understand the construction of Subsection \ref{sec}, and to Alexei Skorobogatov for his interpretation of the proof given in Subsection \ref{alt}.
Most likely Goncharova knew an elementary proof of the result. Unfortunately, she had not written it down, and we were unable to restore it from her notes. However,  switching to algebraic geometry we have found another proof.

The geometric counterpart of this problem is also a complete intersection of quadrics, this time a surface. Just as in the first part, the defining 3-dimensional pencil of quadrics in $\P^5$  enjoys a basis of rank 3 quadrics, and our surface $S$ is not just a K3 surface, but a Kummer surface with a pencil of elliptic curves.
Point counting on $S$ over a finite field with $p$ elements reduces to point counting on 4 twists of an elliptic curve with complex multiplication.

We would like to thank the referee for the careful reading and useful comments.

\part{Consecutive quadratic residues and curves}

\section{Quadratic residue patterns and algebraic curves}\label{s.arithm}
Let $p$ be an odd prime.
Consider the sequence $1$, $2$, \ldots, $p-1$.
Replace every number $i$ in the sequence by the letter $R$ if $i$ is a quadratic residue modulo $p$, and by the letter $N$ otherwise.
Denote by $W_p$ the resulting word.
\begin{defin}
Let $S$ be a word of length $\ell\le{p-1}$ that contains only $R$ and $N$.
Define $n_p(S)$ as the number of sub-words of $W_p$ that coincide with $S$, and are formed by $\ell$ consecutive elements of $W_p$. The word $$\underbrace{R\ldots R}_{l}$$  
will be also denoted by $R^\ell$.
\end{defin}
\begin{example}
Let $p=17$.
Then $W_p=RRNRNNNRRNNNRNRR=R^2NRN^3R^2N^3RNR^2$.
If $\ell=3$, then $n_p(S)=2$ for all words $S$ of length $3$ except for $S= R^3$, and $n_p(R^3)=0$.
\end{example}

Below we will concentrate mainly on the behaviour of  $n_p(R^\ell)$. For other words $S$ of length $\ell$ the problem is similar, and we mostly leave it for future investigators.
\subsection{Early history}
Note that if $\ell=1$, then $n_p(S)=\frac{p-1}2$.
Indeed, $n_p(R)=n_p(N)$, i.e., the number of quadratic residues is equal to the number of non-residues, since the multiplicative group $\F_p^*$ is cyclic of even order, and $n_p(R)+n_p(N)=p-1$. This is known at least since P.~Fermat.

First works known to us dedicated to $\ell\ge 2$ date back to the end of the XIX century. Namely, 
Aladov's paper \cite{A} gives the answer for $\ell=2:$ 
$$n_p(RR)=\frac{p-5}{4},\;n_p(RN)= n_p(NR)= n_p(NN)=\frac{p-1}{4}, \hbox{ if } p=4k+1,$$
$$n_p(RN)=\frac{p+1}{4},\;n_p(RR)= n_p(NR)= n_p(NN)=\frac{p-3}{4}, \hbox{ if } p=4k+3.$$
The proof is completely elementary.
Note that these formulae can be rewritten as 
$$n_p(RR)=\frac{p-4-\genfrac(){0.5pt}{0}{-1}{p}}{4}, \quad n_p(RN)= \frac{p-\genfrac(){0.5pt}{0}{-1}{p}}{4},$$
$$\;n_p(NN)= n_p(NR)=\frac{p-2-\genfrac(){0.5pt}{0}{-1}{p}}{4}$$
for all (odd) $p$. Here and below  $\genfrac(){0.5pt}{0}{a}{p}$ is the Legendre symbol.
 
For $\ell=3$  there are simple explicit formulas for certain linear combinations such as $n_p(RRR)+n_p(NNN)$ \cite{A, St}. Note that those papers treat only the most elementary cases when all answers are certain simple affine-linear functions of $p$.

Jacobstahl's 1906 thesis \cite{J} (citing \cite{St}) explores the case $\ell=3$, which is much trickier.
However, for $p=4k+3$ (and $\ell=3$), the formulas for $n_p(S)$ are quite close to the above ones. Namely, \cite[Part III, Formula I]{J}:
$$n_p(RRR)=n_p(NNN)= n_p(NRR)= n_p(NNR)=\frac{p-3-2\genfrac(){0.5pt}{0}{2}{p}}{8},$$
$$n_p(RRN)= n_p(NRN)= n_p(RNR)=  n_p(RNN)= \frac{p-1+2\genfrac(){0.5pt}{0}{2}{p}}{8}.$$
On the contrary, for $p=4k+1$  the formulas for $n_p(S)$ involve  more complicated ingredients.
Namely, define  $J(k)\in 2\Z$ by  
$$a:=a(p):=J(k)=\sum_{i=1}^{4k-2}\genfrac(){0.5pt}{0}{i(i+1)(i+2)}{p}.$$
Then \cite[Part III, Formula II]{J}:
$$n_p(RRN)= n_p(NRR)= n_p(RNR)=n_p(NNN)= \frac{p-5}{8}-\frac{a}{8},$$
$$n_p(RNN)= n_p(NNR) =\frac{p+1}{8}+\frac{a}{8}, \; $$
$$n_p(RRR)= \frac18\left(p-11-4\genfrac(){0.5pt}{0}{2}{p}\right)+\frac{a}{8},$$
$$n_p(NRN)= \frac18\left(p-3+4\genfrac(){0.5pt}{0}{2}{p}\right) +\frac{a}{8}.$$
 Moreover, one has \cite[Part II]{J}:
 $$p=\frac{a(p)^2}{4} +\frac{b(p)^2}{4}$$
 for
 $$b:=b(p):= \sum_{i=1}^{p}\genfrac(){0.5pt}{0}{i(i^2+s) }{p}, \hbox{ where }  \genfrac(){0.5pt}{0}{c }{p}=0
 \hbox{ for }   p|c,$$
 where $s$ is any non-residue.
This is closely related to the so-called ``Gauss' Last Entry'' (see  Subsection 2.3 below and \cite{M} ).

After Jacobstahl, there was no progress in establishing 
explicit formulas for larger $\ell$ (which is quite natural, see below the end of Subsection \ref{cur}), and during 1920-1930 there were many papers giving different estimations of $n_p(R^\ell),$ in particular, guaranteeing its (strict) positiveness.  All those numerous results are easy consequences of Theorem \ref{est} below.

\subsection{A general formula}
Let then $\ell\ge 1$ be arbitrary,  and let $p\ge 3$ be a prime. 
We still assume that $\ell<p$.
For a word $S=S_1\ldots S_\ell$ as above, define $\varepsilon_i=\varepsilon(S_i)$ by setting $\varepsilon(R)=1$ and $\varepsilon(N)=-1$ for $1\le i\le \ell$.  
Following exposition in \cite{C}, we now recall the classical formula:

\begin{equation}\label{con} 
   n_p(S)=\sum_{j=1}^{p-\ell-1}\prod_{i=1}^{\ell}\frac12\left(1+\eps_i\genfrac(){0.5pt}{0}{i+j -1}{p}\right),   
\end{equation}
which generalizes the above Jacobsthal's formula, see \cite[p.223]{GL} for historical details. 
We are grateful to Maxim Korolev for telling us about this reference.
In particular, formula \ref{con} was used in \cite{Mo} to get the following result:
\begin{thm}\label{est}
    For any $S$ and $p>\ell$ one has
    $$\left|n_p(S)-2^{-\ell}p\right|<(\ell-1)\sqrt{p}+\ell/2.$$
\end{thm}
The proof consists in developing the products in \eqref{con} and applying the  {\em Weil bound} \cite{W} which says that for a polynomial $f\in \F_p$ we have
$$ \left|\sum_{j\in \F_p}  \genfrac(){0.5pt}{0}{f(j)}{p} \right|\le (\deg f-1)\sqrt{p}. $$
In its turn, the  Weil bound is a consequence of the Riemann Hypothesis over finite fields, proved also by Andr\'e Weil in 1940-1941.

That bound for $n_p(S)$ is optimal concerning the power of $p$, but can be sometimes ameliorated for the coefficient at $\sqrt{p}$, see Subsection \ref{cur} below.

In any case \eqref{con} shows that $n_p(S)$ is controlled by some (hyperelliptic) curves and we are going to make them explicit. For simplicity of notation we suppose that $S=R^\ell$, but the same is true {\em mutatis mutandis} for any $S$.
However, below until the  end of Section 2 we suppose that $S=R^\ell$.

\subsection{Curves controlling  $n_p(R^\ell)$}\label{cur}
Since all $\eps_i=1$, for $i=1, \ldots, \ell$, we can re-write \eqref{con} as follows:

$$ n_p(S)=\sum_{j=1}^{p-\ell-1}\prod_{i=1}^{\ell}\frac12\left(1+\genfrac(){0.5pt}{0}{i+j -1}{p}\right)=  $$
 $$=2^{-\ell} \sum_{j=1}^{p-\ell-1}\left(1+\sum_{t=1}^{\ell}\sum_{\substack{T\subseteq[1,l]\\ |T|=t}}\prod_{i\in T}\genfrac(){0.5pt}{0}{i+j -1}{p}\right) = $$ 
 $$=2^{-\ell}  \left(p +\sum_{t=1}^{\ell}\sum_{\substack{T\subseteq[1,l]\\ |T|=t}}\sum_{j=1}^{p}\prod_{i\in T}\genfrac(){0.5pt}{0}{i+j -1}{p}\right) +c_p(\ell)= $$
  $$=2^{-\ell}  \left(p +\sum_{t=1}^{\ell}\sum_{\substack{T\subseteq[1,l]\\ |T|=t}}\sum_{j\in \F_p}\prod_{i\in T}\genfrac(){0.5pt}{0}{i+j -1}{p}\right) +c_p(\ell)= $$
   $$=2^{-\ell}  p +2^{-\ell}\sum_{t=1}^{\ell}\sum_{\substack{T\subseteq[1,l]\\ |T|=t}}N_T +c_p(\ell)$$ 
 for a certain constant $c_p(\ell)\in 2^{-\ell} \Z$ which we will ignore since $c_p(\ell)$ for all $p$ is bounded by a constant $c(\ell)$, which depends only on $\ell$. 
 Note, however, that, if necessary,
 $c_p(\ell)$ can easily be  calculated explicitly. Here we have defined 
 $$N_T:=\sum_{j\in \F_p}\prod_{i\in T}\genfrac(){0.5pt}{0}{i+j -1}{p}=\sum_{j\in\F_p}\genfrac(){0.5pt}{0}{f_T(j)}{p}\in \Z,$$
 for
 $$f_T(X):=\prod_{(i+1)\in T}(X+i)\in\Z[X].$$

 Let then $C_T$ be a hyper-elliptic curve of genus $g_T=g(C_T)$
 $$C_T: y^2=f_T(x), \quad g(C_T)=\lfloor  (t-1)/2\rfloor, \quad t=|T|. $$
 Then it is clear that $N_T=|C_T(\F_p)|-p$, where $|C_T(\F_p)|$ is the number of solutions of $y^2=f_T(x)$ in $\F_p^2$,
 that is, the number of the affine points of $C_T$ (i.e., the total number of points $|\overline{C_T}(\F_p)|$ minus 2 or 0 for an even $t$, and minus 1 for an odd $t$, $\overline{C_T}$ being the smooth projective closure of  $C_T$). Recall, that Weil's bound, i.e. the Riemann hypothesis in characteristic $p$ says that
 $$|  |\overline{C_T}(\F_p)|-p-1|\le 2g_T\sqrt p,$$
 and this implies exactly the estimate of Theorem {\ref{est}}. Recall also that the  number $a_T=p+1-  |\overline{C_T}(\F_p)|$ is called the {\em trace of Frobenius} and it is the trace of the Frobenius operator acting on the group of \'etale $\ell$-adic cohomologies  $H^1(\overline{C_T}, \Q_\ell)$ or its inverse acting on the Tate module $T_\ell(\overline{J_T})$ of the $\overline{C_T}$'s  Jacobian $\overline{J_T}$. Therefore, the $2^\ell-1$ curves $C_T$ control $n_p(R^\ell)$ and also all $n_p(S)$ of length $\ell$.

In fact, there are natural maps of the curve $C_\ell$ of genus $2^{\ell-2}(\ell-3)+1$ corresponding to $n_p(R^\ell)$ onto each $C_T$, and its  Jacobian $J(C_\ell)$ splits up to an isogeny into the product of  Jacobians $J(C_T)$. These maps are given below in the proof of Theorem \ref{jac}.
 
Let us look at small values of $\ell$.

\medskip

$\bf{\ell=2.}$ Here we have three sets $T$, namely, $T_1=\{1\}$, $T_2=\{2\}$, $T_3=\{1,2\}$. 
The corresponding curves
are rational:
$$C_1:y^2=x, \quad C_2:y^2=x+1, \quad C_{1,2}:y^2=x(x+1),$$
hence all traces are zero and one needs only to  calculate $c_p(2)$, which leads to Aladov's result.

\medskip

$\bf{\ell=3.}$ Here we have three possible values of $k=|T|=1,2,3$ and seven sets $T_i$, 3 with  1 element, 3  with 2  elements and the last $T_7=\{1,2,3\}.$ The first 6 give rational curves (conics) and the last one gives  an elliptic curve
$E:y^2=x(x+1)(x+2)$ isomorphic (over $\Z$ and thus over $\F_p$, $p\ge 5$) to $E_0:y^2=x^3-x.$

In particular, Jacobstahl's calculation of the sum 
$$ \sum_{a\in\F_p}\genfrac(){0.5pt}{0}{x(x+1)(x+2) }{p}$$
shows that $ a_0(p)=a_0=J(k)$, $a_0$ being the trace of $E_0$ if $p=4k+1$ (and that $a_0(p)=0$ for $p=4k+3$). 

\medskip

{\bf\em{Remarks. 1.}} There is an interesting historical connection with the famous  Last Entry (of 14.07.1814) in Gauss' {\em Tagebuch}
which says that he had found ``inductively'' the following fact:\smallskip

{\em Let $p$ be a prime  $\equiv  1 \mod 4,$ then the number of solutions to
$$x^2 + y^2 + x^2y^2 \equiv  1 \mod p $$
is $p + 1 - 2a,$
where $p = a^2 + b^2$, and $a$ is odd.}\smallskip

Note that: (1) the sign of $a$ is to be chosen ``appropriately'' ($a-1$ should be divisible by $2+2i$ in the ring of the Gauss integers  $\Z[i]$), and\smallskip

(2) there are four points at infinity included in the solution set.\medskip

The ``lemniscatic'' curve $x^2 + y^2 + x^2y^2=1$ is (birationally) isomorphic to  $E_0$ over $\Q(i)$ and thus over  $\F_p$ with  $p=4k+1$ (it is the Edwards form \cite{E} of $E_0$ over $\Q(i)$ and is singular).  

Therefore, Jacobstahl's proof is essentially the first proof of Gauss' Last Entry theorem, but this was not noted at the time and the first recognized proof was published some 15 years later \cite{Herg}; it uses completely different methods of Complex Multiplication (i.e., the explicit Class field theory of complex quadratic fields). The reader can consult \cite{M},
\cite[Section 3.1]{V},\cite[Section 3.4]{Ro} for some interesting history perspectives  of Gauss' Last Entry, many of them being based on the equality $\mathrm{End}_{\Q(i)}(E_0)=\Z[i]$, $i^2=-1$ i.e., on $E_0$ being a CM (complex multiplication) curve, with  $\mathrm{End}_{ \C}(E_0)=\mathrm{End}_{\overline{\Q}}(E_0)\ne\Z$. In fact, if we consider the natural action of $\Gamma_p:=\mathrm{Gal}(\overline{\F_p}/\F_p)$ on $E_0$, then the image in $\mathrm{End}_{\F_p}(E_0)=\Z[i]$ of the $p$-Frobenius element in $ \Gamma_p$, $\mathrm{Frob}_p:(x,y)\mapsto (x^p,y^p)$,  is just $\omega=a+bi.$\smallskip

{\bf\em{2.}} Thus, $a_0(p)=J(k)$ verifies the relation $J(k)^2+4b^2=4p$ which can be easily expressed as a simple quadratic relation between  $n_p(R^3), b(p)$ and $p.$   One then  naturally asks: is there a hope to get  similar  relations for $n_p(R^\ell), \ell\ge4$? 

The answer is rather  "no"  and we now briefly explain why. As we will see below, $n_p(R^\ell), \ell\ge4$ depends not only on $a_0(p),$  but also on the traces of some elliptic curves without CM.

One notes that for a CM  elliptic curve $E$  the 2-dimensional $\ell$-adic representation $\psi_\ell$ ($\ell $ being an arbitrary prime) of the  Galois group $\Gamma=\mathrm{Gal}(\overline{\Q}/\mathrm{k} )$ for $\mathrm{k}=\mathrm{End}(E)\otimes\Q$
$$ \psi_\ell:\Gamma\to \mathrm{End}(V_\ell(E)),\;V_\ell :=T_\ell(E)\otimes\Q  $$
given by the action on the points of $\ell$-primary torsion of $E$ has an Abelian image which permits to apply the    Complex Multiplication theory, to calculate the Frobenius trace for  a prime $p\ne \ell $ and thus to obtain formulas close to those of Jacobstahl.

On the other hand, for an elliptic curve without CM this image is almost always as large as possible, and thus,  non-Abelian and  even not solvable  which prevents to get  formulas similar to those by Jacobstahl.

\bigskip

 {\bf $\ell$=4.} Here we get 5 non-rational curves,  all elliptic (for $|T|=3$ or $4$):
$$E_0: y^2=x(x+1)(x+2); \quad E_1: \; y^2=x(x+1)(x+3); \quad E_2: \; y^2=x(x+2)(x+3);$$
$$E_3: \; y^2=(x+1)(x+2)(x+3);\quad E_4: \; y^2=x(x+1)(x+2)(x+3).$$
Curves $E_0$ and $E_3$ are isomorphic over $\Q$, and curves
 $E_1$ and $E_2$ are isomorphic over $\Q(i)$, but not over $\Q$,
while $E_0$ and $E_3$ are CM curves, $E_1, E_2$, and $E_4$  are not.

For $p=4k+3$ the curve $E_0$ is {\em supersingular,} i.e., ${a_0(p)}=0$. On the other hand,
since $-1$ is a non-residue in $\F_p$, the curve $E_2$ is a (non-trivial) quadratic twist of $E_1$ 
which implies $a_1(p)+a_2(p)=0.$ Indeed, the non-trivial element $ \sigma\in \mathrm{Gal}(\F_{p^2}/\F_p)$ acts as $-1$ on the Tate
module of $E_1$ and trivially on that of $E_2$. 

Therefore,
\begin{equation}\label{p30}n_p(R^4)=\frac{p}{16}-\frac{a_4(p)}{16} +c_p(4) \;\mbox{for}\; p=4k+3, \end{equation}
\begin{equation}\label{p10}n_p(R^4)=\frac{p}{16}-\frac{a_0(p)}8-\frac{a_1(p)}8-\frac{a_4(p)}{16}+c_p(4) \;\mbox{for}\; p=4k+1, \end{equation}
where $a_i(p)$ is the Frobenius trace of $E_i$ over $\F_p$; in particular, $a_0(p)=J(k)$.
Again, the function $c_p(4)$ of $p$ can be computed explicitly, and is bounded by an absolute constant $c(4)$.

Note that since by the Riemann Hypothesis, $|a_i(p)|<2\sqrt p$ we
get \begin{equation}\label{p3}
    \left|n_p(R^4)-\frac{p}{16}\right|\le \frac{\sqrt p}8+c(4)  \mbox{ for } p=4k+3;\;\end{equation}
    
 \begin{equation}\label{p1}\left|n_p(R^4)-\frac{p}{16}\right|\le \frac{5\sqrt p}8+c(4) \mbox{ for } p=4k+1,\end{equation}
 which is  a certain amelioration of the general estimate. Moreover, as we will see below (Subsection \ref{sta})
those estimates are {\em tight} in a very precise sense.

\subsection{Geometric properties of the  Jacobians}\label{gem}
Let the curve $C_\ell$ be  the following intersection of $\ell-1$ quadrics:
$$x_2^2-x_1^2=1, \quad x_3^2-x_2^2=1, \ldots, \quad x_\ell^2-x_{\ell-1}^2=1$$
and  $C_{p,\ell}\subset\overline{\F}_p^\ell$ be its reduction modulo $p$. By a slight abuse of notation we shall often write $C_{\ell}$ instead of  $C_{p,\ell} $. 
Then every point $(x_1,\ldots,x_\ell)\in C_\ell$ with non-zero coordinates produces the $\ell$-tuple
$(x_1^2,x_2^2,\ldots, x_\ell^2)$ of consecutive quadratic residues modulo $p$.
Taking into account that $x_i$ and $-x_i$ produce the same quadratic residue $x_i^2$ we get 
$$n_p(R^\ell)= 2^{-\ell}{|C_\ell^\circ(\F_p)|},$$
where $C_\ell^\circ\subset C_\ell$ consists of points whose coordinates are non-zero. Therefore,
the curve $C_{\ell}$ controls $n_p(R^\ell)$.
By the adjunction formula, its genus $g_\ell$   is equal to $2^{\ell-2}(\ell-3)+1$.
In particular, $g_2=0,g_3=1,g_4=5$, etc. Fortunately, $g_\ell$ equals the sum of all genera $g_T=g(C_T),T\subseteq[1,\ell],$ which can be written as 
$$g_\ell=2^{\ell-2}(\ell-3)+1=\sum_{j\ge 3}^\ell\left(j-3+\frac12\left(1+(-1)^j) \right)\right) {\ell\choose j}=\sum_{T\subseteq[1,\ell]}g_T.$$
It is not  difficult to prove this equality by induction   for any $\ell\ge 3.$
 
Since the Frobenius trace of $C_{p,\ell}$ coincides (in view of the last subsection) with the sum of those on all curves 
$C_T$, it is natural to suppose that the  Jacobian $J_\ell=J(C_\ell)$ is isogenous to the product of all  Jacobians   $J_T=J(C_T)$. It is not immediately clear, since we get the coincidence of traces only over $\F_p,$
 and not over all its (finite) extensions. However, this is true and our nearest goal is to give a   construction implying this decomposition.  
 
 \begin{thm}\label{jac} The  Jacobian $J_\ell=J(C_\ell)$ is isogenous (over $\Q$) to the product of elliptic and hyperelliptic  Jacobians 
$J_T=J(C_T)$ for all $T\subseteq [1, \ell], \mathrm{card} (T)\ge 3$.
  \end{thm}

  {\em Proof.} First we rewrite the equations defining $C_\ell$ as 
$$x_2^2-x_1^2=1, \quad x_3^2-x_1^2=2, \ldots, \quad x_\ell^2-x_1^2=\ell-1. $$
Let $T\subseteq[1,\ell]$, then we set 
$$ x:=x_1^2, \quad y_T:=\prod_{j\in T}x_j.$$
The map
$$\phi_T: (x_1,x_2,\ldots,x_l)\mapsto (x,y_T)$$
defines a surjective morphism $\phi_T: C_\ell\longrightarrow C_T$. 
When $\mathrm{card} (T)\ge 3$, the genus $g_T := g(C_T)\ge 1$, and $\phi_T$ gives us a map of  Jacobians $\tilde{\phi}_T :J_{C_\ell}\longrightarrow J_{C_T}$. 

Since one has
$$\dim J(C_\ell)=g_\ell= \sum_{T\subseteq[1,\ell]} g_T,$$
if we prove that the product of maps $\tilde{\phi}=\prod_{T\subseteq[1,\ell], \mathrm{card} (T)\ge 3}\tilde{\phi}_T$ is a surjective map of  Jacobians, then  it is an isogeny.

Indeed, the space of regular differential forms on $C_T$ is generated by $x^i dx/y_T,$ $ i=0, \dots, g_T - 1$. Hence
$$\phi_T^*\left(\frac{x^i dx}{y_T}\right)= \frac{2x_1^{2i+1}dx_1}{\prod_{j\in T}x_j }$$ are regular forms on $C_\ell$. For different $T$ these forms are linearly independent, and their total number is $g_{C_\ell}$. Therefore, the differential of $\tilde{\phi}$ at zero is surjective, hence the map $\tilde{\phi}$ is dominant, and we are done. 
 
\qed

  For example, for $\ell=5$ we have in the notation of Subsection \ref{cur} :
  $$\;C_{\{1,2,3\}}=E_0,\;
  C_{\{1,2,4\}}=E_1,\;C_{\{1,3,4\}}=E_2,\;C_{\{2,3,4\}}=E_3, \;  C_{\{1,2,3,4\}}=E_4.$$\medskip

{\bf\em Remark.} There is another, more geometric way to construct the maps $\phi_T$ when $C_T$ is elliptic.
We consider the projective closure $\overline{C}_\ell\subset \mathbb{P}^\ell$ of  $C_\ell$
given by
$$x_2^2-x_1^2=x_0^2, \quad x_3^2-x_1^2=2x_0^2, \quad\ldots, \quad x_\ell^2-x_1^2=(\ell-1)x_0^2. $$

It is the intersection of $(\ell-1)$ quadrics $Q_i$, $i=0,1,\ldots,\ell-2$ all of  rank 3. Consider the corresponding linear system of quadrics,
$$Q_z=\sum_{i=0}^{\ell-2}z_iQ_i,\; z\in \mathbb{P}^{\ell-2}.$$
The singularity of $Q_z$ is equivalent to the singularity of the $(l+1)\times(l+1)$-matrix $M(Q_z)$ of the corresponding quadratic form, i.e., to the equation $\det M(Q_z)=0$. 
This gives a degree $(\ell+1)$ hypersurface $H_\ell\subset \mathbb{P}^{\ell-2}$ called the {\it determinant variety} of the system.  
Since the matrices $M(Q_i)$ are all diagonal, this hypersurface splits into a union of $(\ell+1)$ hyperplanes. 

In particular, for $\ell=4$, the curve $H_4\subset \mathbb{P}^{2}$ is just a union of 5 lines in the plane. Intersections of these lines correspond to quadrics of rank 3 in the web $Q_z$, while all other points on the lines correspond to those of rank 4, i.e., to quadrics that are cones over non-degenerate quadric surfaces in $\P^3$. The vertex of any such cone depends only on the choice of the line $l$ in $H_4$, and not on the point on $l$. Thus the intersection of quadrics corresponding to points on $l$ can be projected to the intersection of two non-degenerate quadrics in $\mathbb{P}^{3}$.   
Since $H_4$ consists of 5 lines,  we get 5 projections of $\overline{C}_4$ onto intersections of two quadrics in $\mathbb{P}^{3}$,  which are birational to our $E_i$, for $i=0$,\ldots, $4$.

One can repeat this construction for $\ell\ge 5$ obtaining some projections $\phi_T$ but not all of them, only those onto elliptic curves.
\medskip

{\bf\em Remark.} Theorem \ref{jac} exposes a very specific property of the  Jacobian, and we wonder whether it is specific for these particular equations, or rather for the geometry of the curve. Indeed, suppose that a smooth curve in $\P^4$ is given as a complete intersection of 3 quadrics, each of rank 3. If the determinant curve of the corresponding net of quadrics is a union of 5 lines, then the  Jacobian splits up to an isogeny into a product of elliptic curves. But the generic case (as pointed to us by W. Castryck) is the irreducible curve of degree 5 with 3 double points. In this case the  Jacobian need not split. The same should be true for the intersection of $n-1$ quadrics, each of rank 3, in $\P^n$. In the generic case the  Jacobian need not split into a product of hyperelliptic  Jacobians. However, if the determinant variety is a union of hyperplanes, this might be true.    

  \subsection{Statistical properties of  $n_p(R^\ell)$}\label{sta}

  As we have seen in Subsection \ref{cur}, there is no hope to get
  a (more or less) explicit formula for $R^\ell$ for $\ell\ge 4.$ 
  However, it is possible to get a very precise information on its
  statistical behaviour as a function of $p$, at least for 
  $\ell=4,5,6,$ and we are going now to explain this point. The 
  general setting is as follows.

  We say that a sequence $y_p,p$ prime; lying in a set $Y$ is equidistributed with respect to a measure $\mu$ on $Y$ if for any $\mu$-measurable subset $X\subset Y$ 
  $$\lim_{x\rightarrow \infty} {\frac{\mathrm{card} \{p\le x : y_p\in X\}}{\mathrm{card} \{p\le x\}}} = \mu (X) \,.$$ 
  
 Note, that this definition implies (Dirichlet plus Chebotareff)
 that if we fix positive integers $r,d$ such that $1\le r\le d-1$ with $(r,d)=1$, then there exists the measure $\mu_d$ on $Y$ which satisfies for any $X\subset Y$ 
 $$\lim_{x\rightarrow \infty} {\frac{\mathrm{card} \{p\le x, p=md+r, m\in \Z_+ : y_p\in X\}}{\mathrm{card} \{p\le x\}}} = \mu_d (X) ,$$
 and 
 $$\mu=\frac{1}{d}\sum_{1\le r\le d-1, (r,d)=1}\mu_d\,.
$$
In fact, we use below only the cases $d=4, r=1,3$ which give the measures $\mu_1$ and $\mu_3,\mu=\frac{\mu_1}{2}+\frac{\mu_3}{2}.$

 Let $E/\Q$ be an elliptic curve, then for any prime $p$ with good reduction $E_p/\F_p$ modulo $p$ we have 
 $$|E(\F_p)|=p+1-a_p(E), \quad |a_p(E)|<2\sqrt p,$$ 
 $a_p(E)$ being the trace of the $p$-Frobenius operator.
There are two cases: \medskip 

A. $E$  is a CM-curve, that is $\mathrm{End}_{\C}(E)\ne\Z;$
\smallskip

B.  $\mathrm{End}_{\C}(E)=\Z$, that is, $E$ has no complex multiplication.\medskip

We can consider the normalized trace $t_E(p):=\frac{a_p(E)}{2\sqrt p}\in I:= ]-1,1[$. Its behaviour in the cases A and B is very different:\smallskip

A. The sequence  $$T(E):=\{t_E(p): E\;\; \hbox{has a good reduction modulo } p\}=\{t_E(p): p\in 
{\rm {Primes}}\setminus S(E)\},$$
($ S(E)$ being a finite set)  is equidistributed with respect to the following measure $\lambda_{cm}$ on $I$:
\begin{equation}
\label{cm}   
\lambda_{cm}=\frac{\delta_0}{2}+ \frac{\mu_{cm} }{2}\end{equation}
for \begin{equation}\label{cm1}\mu_{cm}=\frac{dt}{\pi\sqrt{1-t^2}}  \end{equation}
$\delta_0$ being the Dirac measure in 0. 
In this case the summand $\delta_0/2$ corresponds to the primes with supersingular reductions of $E$ (that is, inert in the complex quadratic  field $F=\mathrm{End}_{\C}(E)\otimes\Q$ of complex multiplication), while $\mu_{cm}/2$ corresponds to the primes of ordinary reduction (that is, splitting in $F$). Note also that for $F=\Q(\sqrt{-1})$ this means that the subsequence measure corresponding to $P_3:=\{p:p=4m+3\}$ is $\mu_3=\delta_0$,  while for  $P_1:=\{p:p=4m+1\}$ we have $\mu_1=\mu_{cm}$.
%In particular, in this case the density of $\{t_E(p), p\in P_1\}$ in $P_1$ %equals $\frac{1}{\pi\sqrt{1-t^2}}.$
\smallskip

B. The sequence $T(E)$ is equidistributed with respect to the ``semi-circle''
 measure $\mu_{ST}$ on $I$,

\begin{equation}\label{st}
\mu_{ST}:=\frac{2\sqrt{1-t^2}}{\pi}dt
\end{equation} 

The result in the case A follows from the class field theory, Hecke's classical results on the expression of the (global) zeta-function of $E$ via $L$-series with Gr\"ossencharakter, equidistribution results  from Appendix A of \cite{Se} and Chapter XV of \cite{La} ; all the details can be found in Section 2 of \cite{Su}.

The result in the case B is just the Sato--Tate conjecture  stated around 1960 and proved completely in \cite{HSBT},
see the details in Section 2 of \cite{Su} .
\medskip

Let now $\ell=4.$

Recall the following definition: let $\{s_1,s_2,\ldots,s_d\}$
be $d$ real sequences 
$$s_i=\{a_{i1},a_{i2},a_{i3},\ldots\}, i=1,2,\ldots,d$$
where, say, $|a_{ij}|\le 1$ for any $i,j$. Suppose that for each 
$i=1,2,\ldots,d$ the sequence $s_i$ is equidistributed with respect to some probabilistic measure $\mu_{i}$ on $[-1,1].$
Then these sequences are {\em statistically (stochastically) independent} if the sequence $\{{\bf v}_1,{\bf v}_2,{\bf v}_3,\ldots\}\subset[-1,1]^d$ is equidistributed with respect to the product measure $\mu_1\times\mu_2\times\ldots\times\mu_d $ on $[-1,1]^d,$
where $${\bf v}_j=\left(a_{1j},a_{2j},\ldots,a_{dj}\right)\in[-1,1]^d. $$

For $n_p(R^4)$ we get the following result 

\begin{thm}\label{st4} 
{\rm I.} If $p$
runs over the set of primes $p \equiv 3 \mod 4$  then the sequence

$$\frac{1}{\sqrt p} n_p(R^4) - \frac{\sqrt{p}}{16}$$
is equidistributed with respect to the measure $\mu_{ST}=\mu_3.$
\smallskip

{\rm II.} Suppose that the 
sequences $ a_0(p), a_1(p), a_4(p)$ are  statistically independent. Then for  $p$ running over the set of primes $p \equiv 1\mod 4$  the sequence

$$\frac{1}{\sqrt p}n_p(R^4) - \frac{\sqrt p}{16}$$
is equidistributed with respect to the measure
$$ \lambda'_{cm}*\mu'_{ST}*\mu''_{ST}=\mu_1,$$
where $\mu*\lambda$ denotes the convolution of measures,
and  $$\lambda'_{cm}(t)=4 % HERE 4 NOT 2, AND THUS 256 BELOW FOR $h_1$ 
\mu_{cm}(4t), \;  
  \mu'_{ST}(t)= 4 \mu_{ST}(4t),            
\;\mu''_{ST}(t)= 8 \mu_{ST}(8t)\,,$$ with
$\mu_{cm}(t)=\frac{dt}{\pi \sqrt{1-t^2}}, \;  \mu_{ST}(t)= \frac{2\sqrt{1-t^2}dt}{\pi} .   $
\end{thm}
It follows directly from formulas \eqref{p30}, \eqref{p10},\eqref{cm},\eqref{cm1} and \eqref{st}. 
We see then that $$\mu_1= \frac{256h_1(t) dt}{\pi^3},\quad h_1:=\frac{1}{ \sqrt{1-16v^2}}*\sqrt{1-16s^2}*\sqrt{1-64u^2}. $$
Therefore, $h_1$ is the convolution of the functions
$f,g,h$ supported on the respective intervals $\left[-1/4, 1/4\right],\;\;  \left[-1/4,1/4\right],\;\; \left[-1/8, 1/8\right].  $ \smallskip

As for the hypothesis in {\rm II}, it looks quite plausible, moreover, it can be deducted from some cases of the Generalised Sato-Tate conjecture (GST, see \cite{VST}).

\medskip

{\bf Remark.} Note that the support of both $\lambda_{cm}$ and $\mu_{ST}$ equals the whole interval $[-1,1]$ thus
$$\mathrm{supp}(\lambda_{cm}')=\left[-\frac14,\frac14\right],\;\; \mathrm{supp}(\mu'_{ST})= \left[-\frac14,\frac14\right],\;\; \mathrm{supp}(\mu_{ST}'')=\left[-\frac18, \frac18\right],$$
and, therefore, $$\mathrm{supp}(\mu_3)=\left[-\frac18, \frac18\right],\quad \mathrm{supp}(\mu_1)=\left[-\frac58, \frac58\right] \,.$$
Indeed, if 
$$\mathrm{supp}(f)=[a_1,b_1],\;\;\mathrm{supp}(g)=[a_2,b_2],\;\;\mathrm{supp}(h)=[a_3,b_3]$$
then $\mathrm{supp}(f*g*h)=[a_1+a_2+a_3,b_1+b_2+b_3].$

\medskip

Let us compare Theorem \ref{st4} with Theorem \ref{est}. On the one hand, our result concerns only  the majority of primes, statistical approach saying nothing about particular primes and even about any sequence of primes of density 0, while the estimate of Theorem \ref{est} holds for all primes. On the other hand, our theorem is much more precise for typical behaviour. For example, as a consequence we get the following result showing the tightness of estimates \eqref{p3} and \eqref{p1}.

\smallskip
\begin{prop} Under the conditions of Theorem {\rm\ref{st4} } for any $\eps>0$  there exist $4$ primes
$p_1\equiv 1\mod 4$, $p_3\equiv 3\mod 4,$ %\linebreak 
$p'_1\equiv 1\mod 4$, $p'_3\equiv 3\mod 4,$
such that 
$$n_{p_1}(R^4)\ge \frac{p_1}{16}+\left(\frac58-\eps\right) \sqrt{p_1},\;\;n_{p'_1}(R^4)\le \frac{p'_1}{16}-\left(\frac58-\eps\right) \sqrt{p'_1} , $$
$$n_{p_3}(R^4)\ge \frac{p_3}{16}+\left(\frac18-\eps\right) \sqrt{p_3},\;\;n_{p'_3}(R^4)\le \frac{p'_3}{16}-\left(\frac18-\eps\right) \sqrt{p'_3}. $$
    
\end{prop}

This result follows immediately from the above remark.

\medskip
Is it possible to obtain analogous results for $\ell>4$ ?

\bigskip
{\bf\em Remark $\mathbf\ell=5$.} 
 
 For $\ell=5$  the genus is $17$, and the same approach yields 15 elliptic curves and a curve of genus $2$, namely:
$$C_{\{1,2,3,4,5\}}:\; y^2=x(x+1)(x+2)(x+3)(x+4).$$
However, counting points on this curve can be reduced to counting points on elliptic curves, since this curve has a non-hyperelliptic involution 
$$\sigma:  x \mapsto -(x+4), \quad y\mapsto   iy,$$
and its  Jacobian splits (over the base field for $p\equiv 1 \mod 4$ and over its quadratic extension for $p\equiv 3 \mod 4$).    Hence, the statistics for $n_p(S)$ can still be given in terms of the Sato--Tate distribution for elliptic curves.

More precisely, under corresponding independence
conditions we see that for $p\equiv 3 \mod 4$
the sequence 
 $$\frac{1}{\sqrt p}n_p(R^5)-\frac{\sqrt p}{32}$$
 is equidistributed with respect to the measure $\psi$, and for $p\equiv 1 \mod 4$ it is equidistributed with respect to the measure $\psi'$
 where the measures $\psi,\psi'$ are given  by convolution formulas, similar to (but more elaborated than) the 
 convolution formula for $\ell=4.$\medskip

{\bf\em Remark $\mathbf\ell=6$.}

If $\ell=6$, the genus is 49, and we get $35$ elliptic curves and $7$ curves of genus 2. Then some genus 2 curves  $C_T,|T|= 5$ have simple  Jacobians, which conjecturally leads to the equividistribution of  
 $$\frac{1}{\sqrt p}n_p(R^6)-\frac{\sqrt p}{64} $$
 with respect to some  $\psi_i, i=1,...m$ given by more complicated convolutions, constructed from $\lambda_{cm}, \mu_{ST}$ and some 2-dimensional 
  Sato--Tate distributions for curves of genus $2$ (see \cite[Theorem 4.8 and § 4.3]{Su}).\medskip 

{\bf\em Remark $\mathbf\ell=7$.}

For $\ell=7$,  the formulas should be of the same type as for $\ell=6,$ since the only genus 3 curve 
$C_T, |T|=7$ has   a non-hyperelliptic involution 
$$\sigma':  x \mapsto -(x+6), \quad y\mapsto   iy,$$
and its  Jacobian splits into a product of an elliptic curve and the  Jacobian of a curve of genus 2. Note, however, that the dimension of the corresponding  Jacobian is already 129.  \medskip

{\bf\em Remark $\mathbf\ell\ge 8$.}

For $\ell\ge 8$ some genus 3 curves  $C_T,|T|= 7$ have simple  Jacobians, and the situation is more complicated,
but can be, in principle, treated  using some cases of the GST. The dimensions (=genera) there begin with  $321$
for $\ell=8.$

\newpage
\bigskip

\part{Goncharova Problem and a K3 Surface}

\section{Generalizations, graphs and surfaces}\label{gen}

\subsection{First generalization}\label{fir}
As we have seen in Subsections \ref{cur}--\ref{sta} , the problem to compute $n_p(R^\ell)$ becomes more and more difficult when $\ell$ grows. Thus one can 
naturally consider some generalizations of the problem weakening the conditions on the structure of  $\ell$-tuples of 
quadratic residues. 
In what follows, we consider $\ell$-tuples of not necessarily consecutive residues $(r_1,\ldots,r_\ell)$ modulo $p$. 
We do not assume that $r_i$ is a quadratic residue.

The most  natural and naive idea is to replace the condition $r_{i+1}-r_i=1$ by the condition that 
$r_{i+1}-r_i$ is a (non-zero) square. Therefore, $n_p(R^3)$
gets replaced by the number of (non-zero) solutions
 or the system 
 $$x_1^2-x_0^2=x_3^2,\quad  x_2^2-x_1^2=x_4^2, $$
 in $\mathbb{G}_m^5(\F_p)$ which gives us a cone over the intersection $ S_{2,2}$ of two quadrics in $\mathbb{P}^4$.
 The surface $ S_{2,2}$ is anti-canonically embedded
 (it is a degree 4 del Pezzo surface) and thus is rational over a certain extension of the base field, which permits to deduce a simple formula for $ |S_{2,2}(\F_p)|.$ 

 Note that here we can also consider the behaviour 
 of $r_{i+k}-r_i, k\ge 2 $ and demand, e.g., the condition 
 $r_{i+k}-r_i=t^2. $ This last condition  is not  interesting for consecutive residues, since the answer  depends then only on the residue of $p$ modulo some positive integer, for instance $r_{i+2}-r_i=2$ is a quadratic residue for $p\equiv\pm 1 \mod 8$ and is not otherwise. For triples, imposing the third condition $r_{i+2}-r_i=t^2$ leads to the system 
 \begin{equation}\label{K3}
  x_1^2-x_0^2=x_3^2,\quad  x_2^2-x_1^2=x_4^2,\quad x_2^2-x_0^2=x_5^2   
 \end{equation}
in $\mathbb{G}_m^6(\F_p)$ defining a cone over a singular K3
surface $X$ in $\mathbb{P}^5$. Calculating the number of its $\F_p$-points in terms of $J(k)$ for $p=4k+1$ is a form of the Goncharova theorem, see Theorem \ref{t.main} below.

\subsection{Further generalizations}\label{sec}
One can naturally associate a problem of the above type with a
graph (non-oriented for $p=4k+1$, oriented for $p=4k+3$),
namely:\smallskip

{\em Non-oriented case.} For $p=4k+1$, let $\Gamma=(V,E)$ be a non-oriented graph on $\ell$ vertices, $V=\{v_0,\ldots,v_{\ell-1}\}$. For any edge $e=(v_i,v_j)\in E$ we impose the condition $r_i-r_j=y_e^2$ for $i>j$ on variables $r_i$, $i=0,\ldots,\ell-1$, and $y_e$, $e\in E$ which leads to a system of 
$m:=|E|$ quadratic equations in $s:=\ell+m=|V|+|E|$ variables 
giving an intersection $Y_\Gamma$ of $m$ quadrics in ${\A}^s$ over $\F_p$.
Note that the definition of $Y_\Gamma$ depends on an enumeration of the vertices of $\Gamma$, however, all resulting varieties will be isomorphic over $\F_p$ since $p=4k+1$.  
Indeed, $-1$ is a quadratic residue, and $r_i-r_j$ and $r_j-r_i$ are either both quadratic residues or both non-residues.
The dimension of the affine variety $Y_\Gamma$ is $\ell$. 

\begin{example}
Let $\Gamma=K_\ell$ be the complete graph. 
Put $d=(\ell+1)\ell/2$. 
Then $Y_\Gamma\subset  \A^\ell\times\A^d$ is the affine variety defined by the system
$$r_i-r_j=y_{ij}^2, \quad 0\le j<i\le\ell-1.$$
\end{example}

When defining $Y_\Gamma$ we no longer assume that $r_i$ are quadratic residues.
However, if the vertex $v_0$ is connected with all vertices $v_1$, \ldots, $v_{\ell-1}$ by edges $e_1$, \ldots, $e_{\ell-1}$, respectively (that is, $\Gamma$ is a cone over a graph on $\ell-1$ vertices), then $(r_1-r_0)$, \ldots, $(r_{\ell-1}-r_0)$ are squares by definition of $Y_\Gamma$, namely, $r_i-r_0=y_i^2$ where $y_i:=y_{e_i}$. 
In this case, we can also associate with $\Gamma$ a projective variety $X_\Gamma\subset \P^{m-1}$ with coordinates $y_e$, $e\in E$. 
Define $X_\Gamma$ as the intersection of $m-\ell+1$ homogeneous quadrics: for every edge $e=(v_i,v_j)\in E$ such that $0<j<i$ we impose the condition $y_i^2-y_j^2=y_e^2$. 
The dimension of the projective variety $X_\Gamma$ is equal to $\ell-2$. 

\medskip

{\bf\em Remark.} 
Our principal example is the complete graph on 4 vertices $K_4$, which is crucial for other cases. 
It is a cone over $K_3$, and gives the above K3 surface $X=X_{K_4}$. 

Consider the graph $K^-_4$ which has one edge less. 
It is a cone over the unique (up to an isomorphism) graph on 3 vertices with 2 edges.
The corresponding surface $X_{K^-_4}$ is given by eliminating one equation (and one variable) from \eqref{K3}, i.e., it is a projection of $X_{K_4}$. 
It coincides with the surface $S_{2,2}$ considered above so its study is much simpler than that of $X_{K_4}$.

For any $\ell$, throwing out an edge is also a projection and the resulting variety is easier to study than the original one. If we throw out more edges, we get further projections.

\medskip

It is easy to relate the numbers $|Y_\Gamma(\F_p)|$ and $|X_\Gamma(\F_p)|$. 
We will do this in Section \ref{ss.main} for the complete graph $\Gamma=K_4$.  

With every non-oriented graph $\Gamma$ we also associate the number $n_p(\G)$ as follows. 
Let $A = (r_1, \ldots, r_\ell)$ be an $\ell$-tuple of pairwise distinct residues modulo $p$.
We no longer assume that they are consecutive. 
We may assign a graph $\Gamma_A$ to $A$ by the following rule. 
Consider $\ell$ vertices $v_1$,\ldots, $v_\ell$, and connect $v_i$ and $v_j$ by an edge if and only if the difference $r_i-r_j$ is a quadratic residue modulo $p$ 
(since $-1$ is a quadratic residue this condition on $r_i-r_j$ is symmetric in $i$ and $j$). 
In what follows, we identify $\ell$-tuples that can be obtained from each other
by a permutation or an additive translation, that is, we do not distinguish between
$(r_1, \ldots, r_\ell)$ and $(r_{\sigma(1)},\ldots, r_{\sigma(\ell)})$, where $\sigma\in S_\ell$, 
and between $(r_1, \ldots, r_\ell)$ and $(r_1+a, \ldots, r_\ell+a)$, where $a$ is a residue modulo $p$.

\begin{defin}
Define $n_p(\Gamma)$ as the number of all
such $\ell$-tuples $A$ that $\Gamma_A$ is topologically equivalent to $\Gamma$.    
\end{defin}

Note that we impose conditions on $r_i-r_j$ for all pairs $(i,j)$. In particular, we require that $r_i-r_j$ be a quadratic nonresidue if $v_i$ and $v_j$ are not connected by an edge. There is an alternative definition that does not include the latter conditions. Namely, define $n'_p(\Gamma)$ as the number of all such $\ell$-tuples $A$ that $r_i-r_j$ is a quadratic residue if $v_i$ and $v_j$ are connected by an edge in $\G$. 

If $\Gamma=K_\ell$ is the complete graph on $\ell$ vertices, then $n_p(\Gamma)=n'_p(\Gamma)$. Clearly, if we know $n_p(\Gamma)$ for all graphs with $\ell$ vertices we can compute $n'_p(\Gamma)$ and vice versa.
The variety $Y_\Gamma$ can be used to compute $n'_p(\Gamma)$, and hence $n_p(\Gamma)$. 

In Section \ref{ss.main}, we shall use the following description of $n_p(K_\ell)$ for the complete graph $K_\ell$. 
Let $Y_0\subset Y_{K_\ell}$ be the subset where $y_{ij}\ne 0$, $1\le j< i\le\ell.$ 
Then we have 
\begin{prop}  
\begin{equation}\label{K}
    n_p(K_\ell)= 2^{-d}(\ell!)^{-1}p^{-1}|Y_0(\F_p)|.
   \end{equation}
\end{prop}

\begin{proof} This is clear in view of the free actions of the permutation group $S_\ell$ (acting on $\{1,\ldots,\ell\}$), $\{\pm 1\}^d$   (acting on  
$\{y_{ij}\}^d$ by sign changes) and of $\F_p$ (acting by translations $(x_1,\ldots,x_\ell)\mapsto(x_1+a,\ldots,x_\ell+a)$).
\end{proof}

Note also that numbers $n_p(\Gamma)$ can be interpreted in terms of the Paley graph associated with $\F_p$. 
We are grateful to Alexander B. Kalmynin from whom we learned about Paley graphs.
For instance, $n_p(K_\ell)$ is equal to the number of $\ell$-cliques in the Paley graph of $\F_p$ divided by $p$. 
In particular, Theorem \ref{t.main} below gives a formula for the number of $4$-cliques in the Paley graph of $\F_p$ (cf. \cite[Corollary 1.5]{BB}).

{\em Oriented case.} For $p=4k+3$ the setting is the same, except
that  the graph should be oriented, since now the conditions $r_i-r_j=y_e^2$ and $r_j-r_i=y_e^2$ are opposite to each other.

\subsection{Main theorem} \label{ss.main} Goncharova's main result concerns the case of quadruples for\linebreak $p=4k+1$, the condition which we suppose to hold until the end of our paper.
As there are $11$ isomorphism classes of simple graphs with four vertices, there are $11$ numbers $n_p(\Gamma)$ for every $p$. 

\begin{thm}[Goncharova] \label{t.main} Assume that $p=4k+1$.
All functions $n_p(\Gamma)$ (considered as functions of $k$) can be explicitly expressed as polynomials in $k$ and $d(k)$ where
$$d(k)=\frac{J(k)^2-4}{32}.$$

In particular for $\Gamma=K_4$  one has
\begin{equation}\label{main}
 n_p(K_4)=\frac{k(k-1)(k-4)+2kd(k)}{24}.   
\end{equation}  
\end{thm}

Goncharova had an elementary proof of this theorem, however, we were unable to recover the proof of the second, most difficult part of it, from her notes.
It seems that results of a somewhat similar flavor are proved in \cite{BE}.

Trying to understand her notes we started with numerical verification of a version of formula \eqref{main} for all odd primes $p<20000$, namely the formula \eqref{2} in Section \ref{s.geom} below, and much later we found out an algebraic geometry proof that we present in the next section. 

Very recently, a generalization of Theorem \ref{t.main} in terms of 4-cliques in the Paley graphs for $\Z/n\Z$ was proved in \cite[Corollary 1.5]{BB} by elementary combinatorial methods. 
 
\section{Proof: counting points  on elliptic curves and a K3 surface}\label{s.geom}
\subsection{A K3 surface} 
 Let $p=4k+1$ and let $\Gamma=K_4$ be the complete non-oriented graph on $\ell$ vertices. We consider the case $\ell=4$.
First, using \eqref{K} we get
$$n_p(K_4)=  2^{-6}24^{-1}p^{-1}|Y_0(\F_p)|$$
for the system 
$$r_1-r_2=y_{12}^2,\; r_1-r_3=y_{13}^2,\; r_1-r_4=y_{14}^2,\; r_2-r_3=y_{23}^2, \; r_2-r_4=y_{24}^2, \; r_3-r_4=y_{34}^2,\;$$
and $y_{ij}\ne 0$, which defines $Y_0$. 

Let us then put $Y'_0=Y_0\cap \{r_4=0\}.$ It is
given by
$$r_1-r_2=y_{12}^2,\; r_1-r_3=y_{13}^2,\; r_1 =y_{14}^2,\; r_2-r_3=y_{23}^2,\; r_2 =y_{24}^2,\; 
r_3 =y_{34}^2,\;$$
and $y_{ij}\ne 0$. 
Hence, we get
$$ n_p(K_4)=  2^{-6}24^{-1}|Y'_0(\F_p)|.$$
Eliminating then $r_1$, $r_2$ and $r_3$ and redefining $x_0=y_{14}$, $x_1=y_{24}$, $x_2=y_{34}$, $x_3=y_{23}$, $x_4=y_{12}$, $x_5=y_{13}$, we get the system of equations

\begin{equation}\label{surface}
 x_1^2 - x_2^2=x_3^2,\quad x_0^2 - x_1^2 = x_4^2, \quad  x_0^2-x_2^2=x_5^2 ,   
\end{equation}
i.e., the  intersection $\widetilde S(\F_p)$ of 3 quadrics, each of rank 3,  in $\mathbb{A}^6$  and its image $S(\F_p)$ in $\P^5$. 

This surface $S$ is not smooth, it has 16 simple singularities
$$x_1=x_2=x_3=0, \mbox{ and } x_0=\pm x_4=\pm x_5   \mbox{ or } x_0=\pm x_4=\mp x_5;$$
$$x_0=x_1=x_4=0, \mbox{ and } x_2=\pm ix_3=\pm ix_5 \mbox{ or } x_2=\pm ix_3=\mp ix_5;$$
$$x_0=x_2=x_5=0, \mbox{ and } x_1=\pm x_3=\pm ix_4  \mbox{ or } x_1=\pm x_3=\mp ix_4;$$
$$x_3=x_4=x_5=0, \mbox{ and } x_0=\pm x_1=\pm x_2   \mbox{ or } x_0=\pm x_1=\mp x_2.$$
It is a singular K3 surface, and with its 16 singularities it is no wonder  that it is a Kummer surface, see subsection \ref{alt}. We consider its reduction modulo $p$, and we want to count the number $|S^0(\F_p)|$, where $S^0$ is the intersection of $S$ with the torus $\mathbb{G}_m^5$, i.e., we consider only solutions with all $x_i\ne 0$. Note that the answer we need is 
\begin{equation}\label{1/32}
n_p(K_4)={\frac{1}{64}}\cdot {\frac{1}{24}}\cdot(p-1) |S^0(\F_p)|,    
\end{equation} 
where the factor $ (p-1) $ corresponds to the fact  that $\tilde S$ is the cone over $S.$

\begin{lemma}\label{S0toS}
We have
$$|S^0(\F_p)|=|S(\F_p)|- 24p+80,$$
and hence
$$n_p(K_4)=\frac{1}{64}\cdot \frac{1}{24}\cdot (p-1)(|S(\F_p)|-24p+80).$$
\end{lemma}

\begin{proof}
Let us first count the number of points in the affine surface $S_a:=S\cap\{x_0=1\}$.
The complement $D:=S\setminus S_a$ is given by equations
$$x_1^2-x_2^2=x_3^2, \quad -x_1^2=x_4^2, \quad -x_2^2=x_5^2.$$
Since $D$ is a 4-fold ramified cover over a smooth conic with 4 ramification points, we have
$$|D(\F_p)|=4(p-3)+8=4(p-1).$$

We now compute $|(S_a\setminus S^0)(\F_p)|$ using decomposition:
$$S_a\setminus S^0=S_1\sqcup S_2\sqcup S_3,$$
where $S_i\subset S_a$ consists of all points with exactly $(6-i)$ non-zero coordinates.
Note that by definition of $S$ the coordinates $x_0$,\ldots, $x_5$  correspond to 6 edges $E_0$,\ldots, $E_5$ of the graph $K_4$, and the defining equations of $S$ have form
$$x_i^2-x_j^2=x_k^2$$
for triples $\{i, j, k\}$ such that the edges $E_i$, $E_j$ and $E_k$ form a triangle.
In particular, if two of coordinates $x_i$, $x_j$, $x_k$ vanish, then the third one also vanishes.
It follows that $S_4=\varnothing$.

To compute $|S_i(\F_p)|$ it is convenient to reintroduce variables $r_1$,\ldots, $r_4$ corresponding to the vertices of $K_4$.
We have
$$r_1-r_4=x_0^2=1, \quad r_2-r_4=x_1^2, \quad r_3-r_4=x_2^2,$$
$$r_2-r_3= x_3^2, \quad r_1-r_2= x_4^2, \quad r_1-r_3=x_5^2.$$
Then $S_1$ can be decomposed as follows:
$$S_1=\bigsqcup_{\{i,j\}\ne \{1,4\}} S_1\cap\{r_i=r_j\}.$$
It is easy to check that $|(S_1\cap\{r_i=r_j\})(\F_p)|=2^4n_p(RR)=16(k-1)$ in all 5 cases.
Indeed, using symmetries of $K_4$ it is enough to consider just two cases: $\{i,j\}=\{2,3\}$ and $\{i,j\}=\{1,2\}$.
In the first case, $r_2-r_4=(r_2-r_1)+(r_1-r_4)$ so $r_2-r_4$ and $r_2-r_1$ are consecutive quadratic residues.
In the second case, so are $r_3-r_4$ and $r_3-r_1$.
It follows that $|S_1(\F_p)|=5\cdot16(k-1)=20(p-5)$.

Similarly, we have the following decomposition:
$$S_2=(S_2\cap\{r_1=r_2, \, r_3=r_4\})\sqcup (S_2\cap\{r_1=r_3, \, r_2=r_4\}),$$
which implies $|S_2(\F_p)|=2\cdot 2^3=16$.

Finally, we have
$$S_3=(S_3\cap\{r_1=r_2=r_3\})\sqcup (S_3\cap\{r_2=r_3=r_4\}),$$
hence, $|S_3(\F_p)|=2\cdot 2^2=8$.

Combining all calculations together we get:
$$|S^0(\F_p)|=|S(\F_p)|-4(p-1)-20(p-5)-16-8=|S(\F_p)|-24p+80.$$
\end{proof}

\begin{lemma}\label{LidatoS} Formula \eqref{main} is equivalent to the following formula:
\begin{equation}\label{2}
 |S(\F_p)|=(p+1)^2+J(k)^2.
\end{equation}
\end{lemma}

\begin{proof}
Indeed, \eqref{main} reads $$n_p(K_4)=\frac{k(k-1)(k-4)+2kd(k)}{24},\;d(k)=\frac{J(k)^2-4}{32},\;k=\frac{p-1}{4}.$$
Since $p=4k+1$ we have

$$n_p(K_4)=\frac{k(k-1)(k-4)+2k\cdot\frac{J(k)^2-4}{32}}{24}=\frac{k}{16\cdot 24}\cdot\left((p-5)(p-17)+J(k)^2-4\right)= $$
$$=\frac{p-1}{64\cdot 24}\left(p^2-22p+81 +J(k)^2 \right),$$
which equals to
$$\frac{p-1}{64\cdot 24}(|S(\F_p)|-24p+80)$$
if $$|S(\F_p)|=(p+1)^2+J(k)^2.$$
\end{proof}

\begin{lemma}\label{StoX} Surface $S$ is birationally isomorphic over $\F_p$ to the affine surface $X$ given by
\begin{equation}\label{double} z^2=(x^2y^2+1)(x^2+y^2).\end{equation}
Moreover, there is a regular isomorphism between $S\setminus D$ and  $X\setminus D_1$,  where the divisors $D$ on $S$ and $D_1$ on $X$ are given by $\{x_1x_5=0\}$ and $\{xyz=0\}$, respectively. 
\end{lemma}

\begin{proof} Let us work with the affine surface $S_a$ given by $x_5=1$ which is defined (after substituting $x_2\mapsto ix_2$) by the system 
$$  x_1^2 + x_2^2=x_3^2,\quad x_1^2 + x_4^2 = x_0^2, \quad  x_0^2+x_2^2=1.$$
We use the following rational parameterization of the first two quadrics defining $S$:
$$x_2=(x^2-1)s, \quad x_1=2xs, \quad x_3=(x^2+1)s,$$  
$$x_4=(y^2-1)t, \quad x_1=2yt, \quad x_0=(y^2+1)t.$$
As follows from the middle column, $xt^{-1}=ys^{-1}$, so if we introduce a new variable $z:=ys^{-1}=xt^{-1}$ and express $x_0$ and $x_2$ in terms of $x$, $y$ and $z$, then the equation $x_0^2+x_2^2=1$ reads (after a simplification):
$$z^2=(x^2y^2+1)(x^2+y^2),$$ 
which defines a singular K3  surface $X$ in a $3$-space.

This gives  birational maps 
$$\varphi:X\longrightarrow S, \quad \psi: S\longrightarrow X, \quad \psi=\varphi^{-1}$$
given by
\begin{equation}\label{12}
\begin{split}
& x_0=\frac{(y^2+1)x}{z}, \quad  x_1=\frac{2xy}{z}, \quad x_2=\frac{(x^2-1)y}{z},  \\ 
& x_3=\frac{(x^2+1)y}{z}, \quad  x_4=\frac{(y^2-1)x}{z}, \quad x_5=1 \\
\end{split}
\end{equation}
for $\varphi$ and 
$$x=\frac{x_2+x_3}{x_1 }, \quad y=\frac{x_0+x_4}{x_1}, \quad z=\frac{2(x_2+x_3)(x_0+x_4)}{x_1^3}$$
for $\psi$. These maps establish a bijection 
between $X\setminus D_1$ and $S\setminus D$,
where the divisors $D_1$ and $D$ are given by $\{xyz=0\}$ and  $\{x_1x_5=0\}$, respectively. 
\end{proof} 

Now we shall compare the number of points on $S$ and $X$ in order to get the following result.

\begin{lemma}\label{points} 
For any $p$ we have $|S(\F_p)| = |X(\F_p)| - 1$.
\end{lemma}

\begin{proof} Indeed, $$|S(\F_p)|=\big| \big( S\setminus D\big)(\F_p)\big|+\big| D(\F_p) \big|= \big|\big(X\setminus D_1\big)(\F_p)\big| + \big| D(\F_p) \big|,$$
where $D=\{x_1=0\}\cup \{x_5=0\}$, and $D_1=\{x=0\}\cup\{y=0\}\cup\{z=0\}$.

Similarly to the proof of Lemma \ref{S0toS}, we get that both $E_1:=S\cap\{x_1=0\}$ and $E_2:=S\cap\{x_5=0\}$ consist of $4p-4$ points. 
The intersection 
$E_1\cap E_2$
is defined by 
$$x_2^2=x_3^2,\quad x_4^2 = x_0^2,\quad x_0^2=-x_2^2,$$ 
and consists of $8$ points.
Hence, 
$$|D(\F_p)|=|E_1(\F_p)|+|E_2(\F_p)|-|(E_1\cap E_2)(\F_p)|=2(4p-4)-8=8p-16.$$

Put $D_x=X\cap\{x=0\}$, $D_y=X\cap\{y=0\}$, and $D_z=X\cap\{z=0\}$.
Since $D_z$ is defined by
$$(x^2+y^2)(x^2y^2+1)=0.$$ 
we have 
$$\big| D_z(\F_p) \big|=\big| \{x^2+y^2=0\}(\F_p) \big|+\big| \{x^2y^2=-1\}(\F_p) \big|-\big| \{x^2+y^2=x^2y^2+1=0 \}(\F_p) \big|=$$
$$=2p-1+2(p-1)-8=4p-11.$$

Similarly, $|D_x(\F_p)|=|D_y(\F_p)|=2p-1$, and $D_x\cap D_y=D_x\cap D_z=D_y\cap D_z=D_x\cap D_y\cap D_z$. 
We have
$$|D_1(\F_p)|=|D_x(\F_p)|+|D_y(\F_p)|+|D_z(\F_p)|-2|\{x=y=z\}(\F_p)|=$$
$$=2p-1+2p-1+4p-11-2=8p-15.$$
Hence, 
$$|X(\F_p)|-|S(\F_p)|  =\big| D_1(\F_p) \big|-\big| D(\F_p) \big|= (8p-15)-(8p-16)=1.$$
\end{proof}

 Now we can finish the proof of our main result that leads to the proof of the Goncharova theorem.
 
 \begin{thm}\label{surface-curve}
Consider the affine surface $X$ in a $3$-space given by the equation:
\begin{equation}\label{3} z^2=(x^2y^2+1)(x^2+y^2),\end{equation}
and the affine plane curve $E_a$ given by the equation:
\begin{equation}\label{4} y^2=x^3-x. \end{equation}
 
Let $M_p$ and $N_p$, respectively, denote the number of solutions of \eqref{3} and \eqref{4}  modulo a prime $p$. 
%$N_p=p-J(k) $ in the above notation.
Then $M_p$ and $N_p$ are related as follows:
\begin{equation}\label{5} M_p=(p+1)^2+(N_p-p)^2+1=\left\{
\begin{array}{ll}
		(p+1)^2+J(k)^2+1  & \mbox{if } p=4k+1 \\
		(p+1)^2+1  & \mbox{if } p=4k+3
	\end{array}
\right.. \end{equation}
\end{thm}

\begin{proof}
Putting $x=x_1$, $y=tx_1$, $z=y_1x_1$ we  replace the surface $X$ by the surface $X'$ given   by the equation:
       
\begin{equation}\label{6} y_1^2=(t^2x_1^4+1)(t^2+1). \end{equation}
Here we regard $x_1$, $y_1$ and $t$ as coordinates in an affine $3$-space so that (\ref{6}) defines a surface. 
Later we will treat $t$ as a parameter, and (\ref{6}) will define a genus 1 plane curve.

It is easy to check that 
$$|X(\F_p)|=|X'(\F_p)|+p.$$
We now count separately points on $X'\setminus X'_0$ and on $X'_0$ where $X'_0:= X'\cap (\{x_1=0\}\cup\{y_1=0\}\cup\{t=0\})$.

By inclusion--exclusion formula we get that if $p\equiv 1 \pmod 4$, then
$$|X'_0(\F_p)|=|[X'\cap\{x_1=0\}](\F_p)|\ + \ |[X'\cap\{y_1=0\}](\F_p)| \ + \ |[X'\cap\{t=0\}](\F_p)| \ -$$
$$- \ |[X'\cap\{x_1=y_1=0\}](\F_p)| \ - \ |[X'\cap\{x_1=t=0\}](\F_p)| \ =$$
$$= \ (p-1)+(4p-10)+2p-2-2 \ = \ 7p-15$$ 
Here we used that $X'\cap\{y_1=t=0\}$ is empty.

To count points on $X'\setminus X'_0$ we regard $t$ as a parameter.
If $t$ is fixed then equation \eqref{6} defines an elliptic curve $X_t\subset X'$. 
The number of points on such a curve is uniquely determined by the quadratic residue pattern formed by $(t, t^2+1)$. 
Moreover, if the pattern is fixed, the corresponding fibers $X_t$ are isomorphic to each other over $\F_p$.
Hence, there are four cases to consider: $R R$, $R N$, $N R$, $N N$.
There are four elliptic curves $E_1$, $E_2$, $E_3$, $E_4$, respectively. %that parameterize pairs $(t,t^2+1)$ in each case.
For instance, if $(t,t^2+1)$ has pattern $RR$ then $t=s^2$ and $t^2+1=u^2$ where
$(s,u)$ is a point with non-zero coordinates on the plane curve $E_1$ given by the equation $u^2=s^4+1$.
Clearly, the points $(\pm s,\pm u)$  correspond to the same value of $t$.
Hence, the number of values $t$ such that the pair $(t,t^2+1)$ has 
pattern $RR$ is equal to $\frac{1}{4}|E^\circ_1(\F_p)|$ where $E^\circ_1$ denotes $E_1\setminus(\{s=0\}\cup\{u=0\})$.
 
It is interesting that for all values of $t$ parameterized by $E_i(\F_p)$ the corresponding elliptic curve $X_t$ is isomorphic to $E_i$ over $\F_p$.
For instance, if $t$ and $t^2+1$ are quadratic residues, then $X_t$ is isomorphic to the curve $y^2=x^4+1$ by the change of variables $x=sx_1$, $y=u^{-1}y_1$.
The same holds in the other three cases.

Hence, we get the following identity:
\begin{equation}\label{8}
    \left|[X'\setminus X'_0](\F_p)\right|=\frac{1}{4}\sum_{i=1}^{4}\left|E^\circ_i(\F_p)\right|^2 
\end{equation}

Let $a$ be the trace of the elliptic curve $y^2=x^4-1$ over $\F_p$. Then for all quadratic twists of this curve the trace equals $\pm a$.

Using Table 1 we see that for $p=8k\pm1$ the right hand side is equal to
$$\frac{1}{4}[(p-7+a)^2+2(p-3-a)^2+(p+1+a)^2]=p^2-6p+17+a^2.$$
The computation with Table 2 for $p=8k\pm3$ yields the same answer.
Here we use that $E_2$ is the quadratic twist of $E_1$, and $E_4$ is the quadratic twist of $E_3$, hence, their Frobenius traces have opposite signs. 
Note also that the traces of $E$, $E_1$ and $E_3$ are equal up to a sign since
all curves have the CM field $\Q(i)$ with the class number $h=1$ (the ring of integers $\Z[i]$ is a PID).
A direct calculation of $|E_1(\F_p)|$ and $|E_3(\F_p)|$ for $p=5$ shows that their traces have opposite signs.
Hence, if $a$ denotes the trace of $E_1$, then the traces of $E_2$, $E_3$ and  $E_4$ are equal to $-a$, 
$-a$ and $a$, respectively. 
Since the trace of $E$ is equal to $(N_p-p)$, we also have 
$$a^2=(N_p-p)^2.$$

\begin{table}[h]
		\centering
\begin{tabular}{|l|c|c|c|c|} \hline\hline
{\em Curve} & \# points at $\infty$ & \# points at $u=0$ or $s=0$ & sum & trace of Frobenius\\ \cline{2-5}
$u^2=s^4+1$ & 2 & 6 & 8 & $a$\\
\cline{2-5}
$\delta u^2=s^4+1$ & 0 & 4 & 4 & $-a$\\
\cline{2-5}
$u^2=\delta^2s^4+1$ & 2 & 2 & 4 & $-a$\\
\cline{2-5}
$\delta u^2=\delta^2s^4+1$ & 0 & 0 & 0 & $a$\\
\hline\hline
\end{tabular}
\medskip

\caption{Case $p=8k\pm1$}
\end{table}

Here and below $\delta\in\F_p^*$ denotes a quadratic non-residue.

\begin{table}[h]
		\centering
\begin{tabular}{|l|c|c|c|c|} \hline\hline
{\em Curve} & \# points at $\infty$ & \# points at $u=0$ or $s=0$ & sum & trace of Frobenius\\ \cline{2-5}
$u^2=s^4+1$ & 2 & 2 & 4 & $a$\\ 
\cline{2-5}
$\delta u^2=s^4+1$ & 0 & 0 & 0 & $-a$\\
\cline{2-5}
$u^2=\delta^2s^4+1$ & 2 & 6 & 8 & $-a$\\
\cline{2-5}
$\delta u^2=\delta^2s^4+1$ & 0 & 4 & 4 & $a$\\
\hline\hline
\end{tabular}
\medskip

\caption{Case $p=8k\pm3$}
\end{table}

Let $E_0$ be the curve given by $y^2=x^3-x$, and $a_0$ be its trace. 
This curve enjoys complex multiplication by $\Q(i)$, just as the curve $y^2=x^4-1$. Therefore, they are isogenous over $\F_{p^2}$, and $a^2=a_0^2$. If $p=4k+3$, the curve $E$ is supersingular, and thus $a=a_0=0$. For $p=4k+1$, a direct calculation shows that for $p=8k+1$ they are even isomorphic over $\F_{p}$, and $a=a_0$, whence for $p=8k+5$ they are isomorphic only over $\F_{p^2}$, and $a=-a_0$.

To prove the first equality in \eqref{5} we combine the above formulas involving $X$, $X'$ and $X_0$. 
We get
$$|X(\F_p)|=|X'(\F_p)|+p=|[X'\setminus X'_0](\F_p)|+|X_0'(\F_p)|+p=$$
$$=|[X'\setminus X'_0](\F_p)|+(7p-15)+p=(p^2-6p+17+a^2)+(7p-15)+p$$
$$=(p+1)^2+a^2+1=(p+1)^2+(N_p-p)^2+1.$$

Now let us prove the second equality in \eqref{5}.

The curve $E_0:y^2=x^3-x$ is isomorphic to $E_1 : y^2=x(x+1)(x+2)$. 
Thus, Jacobstahl's sum for $p=4k+1$ is related to $N_p$ as follows:
$$J(k)=\sum_{a\in\F_p}\genfrac(){0.5pt}{0}{x(x+1)(x+2) }{p}=a_0=p-N_p,$$
$a_0$ being the trace of $E_0$, and thus of $E_1$, and $N_p$ as above being the number of $\F_p$-points on the affine curve $y^2=x^3-x$.
\end{proof}

Recently, Alexey Ustinov (HSE University) found a short elementary proof of Theorem \ref{surface-curve} via algebraic manipulations with Jacobstahl's sums \cite{U}.

And now the last 
\begin{lemma}\label{equiv} Formula \eqref{2} is equivalent to formula \eqref{5}. \end{lemma}

\begin{proof}
 
Lemma \ref{points} and \eqref{5} give
$$\vert S(\F_p)\vert = M_p -1 = (p+1)^2+J(k)^2+1 -1=   (p+1)^2+J(k)^2.$$
\end{proof}

Now we are ready to prove the Goncharova Theorem.
\vskip 0.2 cm
{\em{Proof of Theorem {\rm{\ref{t.main}}}}.}

Just combine Theorem \ref{surface-curve} with Lemmas \ref{LidatoS} and \ref{equiv}. The second (difficult) part of the Goncharova Theorem \ref{t.main} concerning the complete graph $K_4$ follows immediately. The case of other graphs is much easier, the corresponding surfaces being rational, and  we leave its proof to the reader. 
\qed

\subsection{Alternative method}\label{alt}
One can look at the surface $X$ from another point of view, which is less elementary and more geometric, and leads to an alternative proof of formula \eqref{8}. We keep the notation from the proof of Theorem \ref{surface-curve}.   

Let $E_1$ be given by $u^2=t^4+1$. The associate Kummer surface $K=\mathrm{Kum}(E_1\times E_1)$ is isomorphic to $u^2=(r^4+1)(v^4+1).$

 Look now at the   equation \eqref{6}, i.e.,  $y^2=(t^2x_1^4+1)(t^2+1)$, which is a singular K3 surface $X’$. 
 It is birational to $X$ and to the intersection of three quadrics $S$. Write $u=y$, $t=v^2$ and $r=vx_1.$ This gives a dominant rational map of degree 2 from $K$ to $X’$. So one has a dominant rational map of degree 4 from $E_1 \times E_1$ to $X’$. Generically it’s a Galois cover with group $(\Z/2\Z)^2$ given by extracting square roots of $t$ and $t^2+1$. Geometrically, this map is the quotient map with respect to the action of $(\Z/2\Z)^2$, where one generator is the antipodal involution $(P,Q) \mapsto (-P,-Q)$ and the other generator is the translation $(P,Q)\mapsto (P+R,Q+R)$ by a point $R$ of order 2.

 Therefore one has a $(\Z/2\Z)^2$-torsor $E_1^{\circ}\times E_1^{\circ}\to X'\setminus X'_0$ given by simultaneously changing the sign of $u$ and $t$ on both copies of $E_1$.

 We can twist this torsor by a pair of elements of $\F_p^*.$ Up to isomorphism, this depends only on the classes modulo squares, and since   $\F_p^*/\F_p^{*2} \simeq \Z/2\Z$, we get four elliptic curves $E_i$ given by $u^2=a(b^2t^4+1),$ where $a$ and $b$ are representatives of cosets of $\F_p^*\mod \F_p^{*2}$. Any $\F_p$-point of $X'\setminus X'_0$ lifts to four distinct points on $E_i\times E_i$ for exactly one $i,$ which implies  formula \eqref{8}.

{\bf\em Remarks.} 

1) The above argument shows also that $X'$ is birational to the Kummer surface of the Abelian surface which is the quotient of $E_1\times E_1$ by the subgroup of order 2 generated by $(R,R)$, where $R$ is a point of order 2 on $E_1$. The translation by $R$ sends $(u,t)$ to $(-u,-t).$

2) Note also that $E_1$ is isomorphic to $y^2=x^3-4x$ over $\Q$, which is isomorphic to $y^2=2(x^3-x),$ so $E_1$  is the quadratic twist by 2 of $E_0:y^2=x^3-x.$ 
One verifies that the (elliptic) involution of $E_1$ that changes the sign of $u$ and preserves $t$ is the same as the involution that changes the sign of $y$ and preserves $x,$ so $K$ is also isomorphic to $y^2=(x^3-x)(z^3-z)$.

\end{document}